\documentclass[11pt,oneside]{article}
\usepackage[latin1]{inputenc}
\usepackage[T1]{fontenc}
\usepackage{ae}
\usepackage{aecompl}
\usepackage{amsmath}
\usepackage{amsthm}
\usepackage{amsfonts}
\usepackage{amssymb}
\usepackage{mathrsfs}
\usepackage[dvips]{graphicx}
\usepackage{float}
\usepackage[all]{xy}
\usepackage{color}
\theoremstyle{plain}
\newtheorem{prop}{Proposition}[section]
\newtheorem{teo}[prop]{Theorem}

\newtheorem{defn}[prop]{Definition}

\newtheorem{cor}[prop]{Corollary}
\newtheorem{lem}[prop]{Lemma}

\theoremstyle{remark}
\newtheorem{oss}[prop]{Remark}
\definecolor{red}{rgb}{1,0,0}
\definecolor{green}{rgb}{0,1,0.2}

\title{Families of nodal curves in $\mathbb{P}^r$ \\ with the expected number of moduli}
\author{Edoardo Ballico, Luca Benzo, Claudio Fontanari}
\date{}
\begin{document}
\maketitle
\begin{abstract}
Let $\mathscr{V}^{r}_{d,g, \delta}$ be the Hilbert scheme of nodal curves in $\mathbb{P}^r$ of degree $d$
and arithmetic genus $g$ with $\delta$ nodes. Under suitable numerical assumptions on $d$ and $g$, for every
$0 \le \delta \le g$ we construct an irreducible component of $\mathscr{V}^{r}_{d,g, \delta}$ having the
expected number of moduli.
\end{abstract}

\begin{section}{Introduction}

The notion of \emph{family of curves with the expected number of moduli} was introduced into the modern literature
by Edoardo Sernesi in his seminal paper \cite{Sernesi}. Roughly speaking, it refers to any irreducible component
of a suitable Hilbert scheme of projective curves such that the image of the natural projection to the corresponding
moduli space defined by forgetting the extrinsic data has the expected dimension.

In the case of nodal plane curves, the results presented in \cite{Sernesi} (see Theorem 4.2) are quite complete and
they have been later extended to nodal curves on smooth surfaces of general type by Flaminio Flamini in \cite{Flamini}
and to plane curves with both nodes and cusps by Concettina Galati in \cite{Galati}. The results obtained in \cite{Sernesi}
for smooth curves in $\mathbb{P}^r$ with $r \ge 3$, instead, were far to be sharp and since then they have been improved
in a series of contributions by several authors. For $r=3$ we are referring to \cite{Pareschi} by Giuseppe Pareschi
and to \cite{Walter} by Charles H. Walter (indeed, see the paragraph just after the statement of Theorem 0.3 on p. 303),
while for $r \ge 4$ to \cite{BallicoEllia} by the first author and Philippe Ellia and to \cite{Lopez} and \cite{Lopezbis} by Angelo F. Lopez. The method applied by Lopez is admittedly
the same as Sernesi's one (essentially, the smoothing of a reducible curve) and subsequent improvements are gained just
by building on clever and clever curves.

Here we address the case of nodal curves in $\mathbb{P}^r$ with $r \ge 3$ by careful partial smoothings of reducible
curves with rational components. Namely, let $V$ be an irreducible component of the locally closed subscheme $\mathscr{V}^r_{g,d,\delta}
\subset Hilb^r_{g,d}$ parameterizing stable curves with (exactly) $\delta$ nodes in $\mathbb{P}^r$, of arithmetic genus $g$ and degree $d$.
The universal family $\mathscr{G} \rightarrow V$ defines in a functorial way a morphism $\pi:V \rightarrow \overline{\mathcal{M}}^{\delta}_{g}$, where $\overline{\mathcal{M}}^{\delta}_{g} \subset \overline{\mathcal{M}}_{g}$
is the closure in $\overline{\mathcal{M}}_{g}$, the coarse moduli space of stable curves of genus $g$, of the subvariety parameterizing curves with (exactly) $\delta$ nodes. Hence we may formally introduce the following:

\begin{defn}\label{defnumbmod2}
The \emph{number of moduli of} $V$ is $\dim \pi(V)$. We say that $V$ has the \emph{expected number of moduli} if $\dim \pi(V)=\min \left\{3g-3-\delta,3g-3+\rho(g,r,d)-\delta\right\}$, where $\rho(g,r,d)= g - (r+1)(g-d+r)$ is the
Brill-Noether number.
\end{defn}
Our main result is the following special case of Theorem \ref{main}:
\begin{cor}
Let $r \geq 3$ and $d \geq (2r+2)r$ be integers. For all integers $g$ and $\delta$ such
that $d+\lfloor \frac{d}{r} \rfloor-r \leq g \leq d+ \lfloor \frac{d}{r} \rfloor$
and $0 \leq \delta \leq g$, there exists an irreducible nondegenerate nodal curve
$C_{\delta} \subset \mathbb{P}^r$ of arithmetic genus $g$ with $\delta$ nodes, such that
$[C_{\delta}]$ is a smooth point of an irreducible component $V \subset
\mathscr{V}^{r}_{d,g, \delta}$ having the expected number of moduli.
\end{cor}

Notice that our upper bound on $g$ turns out to be not too far from optimal for components $V$ such that the general point $[C] \in V$ satisfies $h^1(C, N_C)=0$, where $N_C$ denotes the normal sheaf of $C$ in $\mathbb{P}^r$, since in this case $\chi (N_C) = (r+1)d-(r-3)(g-1) \ge 0$.
We also point out that the above statement improves a previous weaker result on the existence of
regular families of nodal projective curves due to the first author and Luca Chiantini
(see \cite{BC}, Theorem 1.3).

This research project was partially supported by GNSAGA of INdAM, by PRIN 2010--2011
"Geometria delle variet\`a algebriche",
and by FIRB 2012 "Moduli spaces and Applications".

We work over an algebraically closed field $k$ of characteristic zero.\end{section}
\begin{section}{Background results}
Let $C$ be a projective connected reduced nodal curve of arithmetic genus $g$. For every divisor $D$ on $C$ one can consider the cup-product map
$$\mu_0(D):H^{0}(D) \otimes H^{0}(\omega_C(-D)) \rightarrow H^{0}(\omega_C)$$
where $\omega_C$ is the dualizing sheaf on $C$. If $h^{0}(D)=m+1$ and $\deg D=d$, an obvious necessary condition for $\mu_0(D)$ to be surjective is that the Brill-Noether number $\rho(g,m,d)=g-(m+1)(g+m-d)=h^{0}(\omega_C)-h^{0}(D)\cdot h^{0}(\omega_C(-D))$ is non-positive.\\
If $C \subset \mathbb{P}^r, r \geq 2$, there is a 4-terms exact sequence
$$0 \rightarrow T_C \rightarrow {T_{\mathbb{P}^r}}_{|C} \rightarrow N_C \rightarrow T^1_C \rightarrow 0$$
where $T_C \doteq \textbf{Hom}(\Omega^1_C,\mathcal{O}_C)$ and $T^1_C$ is the \emph{first cotangent sheaf} of $C$, which is a torsion sheaf supported on the singular locus $S$ of $C$ and having stalk $\mathbb{C}$ at each of the points (see \cite{S}, Proposition 1.1.9 (ii)). The sheaf $N'_C \doteq \ker \left\{N_C \rightarrow T^1_C \right\}$ is called the \emph{equisingular normal sheaf} of $C$.\\
For each $p \in S$ let $T^1_p$ denote the restriction to $p$ of $T^1_C$ extended by zero on $C$.\\
From the deformation-theoretic interpretation of the cohomology spaces associated to these sheaves, it follows that, if the cohomology maps $H^{0}(N_C) \rightarrow H^{0}(T^1_p)$ are surjective for every $p \in S$ and $H^{1}(N_C)=(0)$, then the curve $C$ (flatly) deforms to a smooth curve in $\mathbb{P}^r$ i.e. it is \emph{smoothable in $\mathbb{P}^r$}. Note that in particular $C$ is smoothable if $H^{1}(N'_C)=(0)$ (see \cite{Sernesi}, Section 1 for an extensive discussion of all this topic).
We recall the following well-known result:
\begin{prop}
\label{vgddelta}
Let $C \subset \mathbb{P}^r$ be a stable curve of degree $d$ and arithmetic genus $g$ with $\delta$ nodes and let $N'_C$ be the equisingular normal sheaf of $C$ in $\mathbb{P}^r$. One has
\begin{itemize}
\item[(i)] $T_{[C]}\mathscr{V}^{r}_{g,d,\delta} \cong H^{0}(N'_C)$;
\item[(ii)] $\chi(N'_C) \leq \dim_{[C]}\mathscr{V}^{r}_{g,d,\delta} \leq h^{0}(N'_C)$;
\item[(iii)] if $H^{1}(N'_C)=(0)$, then $\mathscr{V}^{r}_{g,d,\delta}$ is smooth of dimension $h^{0}(N'_C)$ at $[C]$. Moreover, for every $0 \leq \delta' \leq \delta$ there exists a deformation $C'$ of $C$ which is a nodal curve with (exactly) $\delta'$ nodes.
\end{itemize}
\end{prop}
\begin{proof}
$(i)$, $(ii)$ and the first part of $(iii)$ follow from the fact that $H^{0}(N'_C)$ and $H^{1}(N'_C)$ are, respectively, the tangent and an obstruction space for the functor of locally trivial embedded deformations of $C$ in $\mathbb{P}^r$ (see again \cite{Sernesi}, Section 1 for details). For the second part of $(iii)$, the cohomology sequence associated to the exact sequence $0 \rightarrow N'_C \rightarrow N_C \rightarrow T^1_C \rightarrow 0$, where $T^1_C$ is the first cotangent sheaf of $C$, gives the surjectivity of the map $H^{0}(N_C) \rightarrow H^{0}(T^1_C)$. Hence for every subset of $\delta' \leq \delta$ nodes of $C$, there exists a deformation of $C$ which is locally trivial at these nodes and smooths the others.
\end{proof}
\begin{prop}
\label{condizioniperexpected}
Let $C \subset \mathbb{P}^r$ be a nondegenerate stable curve of degree $d$ and arithmetic genus $g$ with $\delta$ nodes. Assume that the following conditions are satisfied:
\begin{itemize}
\item[(i)] $H^{0}(\mathcal{O}_C(1))=r+1$ i.e. $C$ is linearly normal;
\item[(ii)] $\mu_0(\mathcal{O}_C(1))$ has maximal rank;
\item[(iii)] $H^{1}(N'_C)=(0)$.
\end{itemize}
Then $C$ is parameterized by a smooth point of an irreducible component $V \subset \mathscr{V}^r_{g,d,\delta}$ having the expected number of moduli.
\end{prop}
\begin{proof}
Since $H^{1}(N'_C)=(0)$, by Proposition \ref{vgddelta} the curve $C$ is parameterized by a smooth point of an irreducible component $V \subset \mathscr{V}^r_{g,d,\delta}$.\\
By \cite{S}, Theorem 2.4.1, the vector space $H^{1}(T_C)$ parameterizes isomorphism classes of first-order locally trivial deformations of $C$. The coboundary map $\kappa_C: H^{0}(N'_C) \rightarrow H^{1}(T_C)$ in the cohomology sequence of the exact sequence
$$0 \rightarrow T_C \rightarrow {T_{\mathbb{P}^r}}_{|C} \rightarrow N'_C \rightarrow 0$$
has rank $rk(\kappa_C)=3g-3-\delta-h^{1}({T_{\mathbb{P}^r}}_{|C})$. Since $C$ is linearly normal, the Euler sequence restricted to $C$ writes as
$0 \rightarrow \mathcal{O}_C \rightarrow H^{0}(\mathcal{O}_C(1))^{*} \otimes \mathcal{O}_C(1) \rightarrow {T_{\mathbb{P}^r}}_{|C} \rightarrow 0$.
It follows that $H^{1}({T_{\mathbb{P}^r}}_{|C}) \cong \ker \mu_{0}(\mathcal{O}_C(1))^{*}$ and therefore, since $\mu_{0}(\mathcal{O}_C(1))$ has maximal rank, $rk(\kappa_C)=\min \left\{3g-3-\delta,3g-3+\rho(g,r,d)-\delta \right\}$.\\
Since properties $(i)$ and $(ii)$ are open in $V$, the rank of the map $\kappa_C$ is constant on an open subset of $V$ and thus one has that $rk(\kappa_C)=\dim \pi(V)$. Therefore $V$ has the expected number of moduli.
\end{proof}
\begin{prop}
\label{condizioniequivalentialinearmentenormaleerangomassimo}
Let $C \subset \mathbb{P}^r$ be a connected reduced nondegenerate nodal curve of degree $d$. The following are equivalent:
\begin{itemize}
\item[(i)] $H^{0}(\mathcal{O}_{C}(1))=r+1$ and $\mu_0(\mathcal{O}_C(1))$ is surjective;
\item[(ii)] $h^{0}({T_{\mathbb{P}^r}}_{|C})=(r+1)^2-1$;
\item[(iii)] The restriction map $H^{0}(T_{\mathbb{P}^r}) \rightarrow H^{0}({T_{\mathbb{P}^r}}_{|C})$ is an isomorphism.
\end{itemize}
\end{prop}
\begin{proof}
The proof goes on exactly as the one of \cite{Pareschi}, Proposition 1.1.1.
\end{proof}
Hence in order to find a component of $\mathscr{V}^r_{g,d,\delta}$ with the expected number of moduli it is sufficient to find a nondegenerate stable curve $C \subset \mathbb{P}^r$ with $\delta$ nodes such that $H^1(N'_C)=0$ and $h^{0}({T_{\mathbb{P}^r}}_{|C})=(r+1)^2-1$.\\
Let $X \subset \mathbb{P}^r$ be a connected reduced locally complete intersection curve, let $p_1,...,p_k$ be smooth points of $X$ and let $T \doteq \left\{p_1,...,p_k\right\}$.
Let $\mathcal{A}$ be the category of local artinian $k$-algebras, let $A \in \mathcal{A}$ and consider the two covariant functors
$$H^{T}_{X}, H^{\prime T}_{X}:\mathcal{A} \rightarrow \text{Sets}$$
defined as
$$H^{T}_{X}(A) \doteq \left\{\text{deformations of } X \text{ in } \mathbb{P}^r \text{ over }A \text{ keeping } T \text{ fixed}\right\}$$
$$H^{\prime T}_{X}(A) \doteq \left\{\text{locally trivial deformations of } X \text{ in } \mathbb{P}^r \text{ over } A \text{ keeping } T \text{ fixed}\right\}.$$
The \emph{tangent spaces to the functors} $H^{T}_{X}, H^{\prime T}_{X}$ are by definition the spaces of the respective first-order deformations i.e. $H^{T}_{X}(k[\epsilon])$ and $H^{\prime T}_{X}(k[\epsilon])$, where $k[\epsilon]$ is the ring of dual numbers.
\begin{lem}
\label{nc-t}
Let $X \subset \mathbb{P}^r$ be a connected reduced locally complete intersection curve, let $p_1,...,p_k$ be smooth points of $X$ and let $T \doteq \left\{p_1,...,p_k\right\}$. Then one has $H^{T}_{X}(k[\epsilon]) \cong H^{0}(N_X(-T))$ and $H^{\prime T}_{X}(k[\epsilon]) \cong H^{0}(N'_X(-T))$. Moreover $H^{1}(N_X(-T))$ (respectively, $H^{1}(N'_X(-T))$) is an obstruction space for the functor $H^{T}_{X}$ (respectively, $H^{\prime T}_{X}$).
\end{lem}
\begin{proof}
Let $\epsilon:\widetilde{\mathbb{P}}^r \rightarrow \mathbb{P}^r$ be the blow-up at $T$, let $E$ be the exceptional divisor and let $\widetilde{X}$ be the strict transform of $X$ with respect to $\epsilon$.
Consider the commutative exact diagram
\begin{equation}
\xymatrix{ & 0 \ar[d] & 0 \ar[d] & 0 \ar[d] &\\ 0 \ar[r] & T_{\widetilde{X}} \ar[r] \ar[d]^{\cong} & {T_{\widetilde{\mathbb{P}}^r}}_{|\widetilde{X}} \ar[r] \ar[d] & N'_{\widetilde{X}/\widetilde{\mathbb{P}}^r} \ar[r] \ar[d] & 0\\ 0 \ar[r]  & {\epsilon^{*}T_X}_{|\widetilde{X}} \ar[r] \ar[d] & {\epsilon^{*}T_{\mathbb{P}^r}}_{|\widetilde{X}} \ar[r] \ar[d] & {\epsilon^{*}N'_X}_{|\widetilde{X}} \ar[r] \ar[d]^{\gamma} & 0 \\ & 0 \ar[r] & \mathcal{O}_E(-E)_{|\widetilde{X}} \ar[r] \ar[d] & \mathcal{O}_E(-E)_{|\widetilde{X}} \ar[d] & \\ & & 0 & 0 &}
\end{equation}
One has that $N'_{\widetilde{X}/\widetilde{\mathbb{P}}^r} \cong \ker \gamma \cong {\epsilon^{*}N'_X}(-E)_{|\widetilde{X}} \cong N'_X(-T)$, where the last isomorphism holds by the projection formula. Analogous considerations give $N_{\widetilde{X}/\widetilde{\mathbb{P}}^r} \cong N_X(-T)$. The statement immediately follows.
\end{proof}
\begin{lem}
\label{sequences}
Let $C_1 \subset \mathbb{P}^r$ be a connected reduced nodal curve and $C_2 \subset \mathbb{P}^r$ be a smooth connected curve such that $C \doteq C_1 \cup C_2$ is nodal and $\Delta \doteq C_1 \cap C_2$ is a smooth $0$-dimensional subscheme of $C_1$ and $C_2$ supported at smooth points of $C_1$. Then there exist exact sequences:
\begin{equation}
\label{seq1}
0 \rightarrow \mathcal{I}_{C_1/C} \otimes N_{C} \rightarrow N'_C \rightarrow N'_{C_1} \rightarrow 0
\end{equation}
\begin{equation}
\label{seq2}
0 \rightarrow N_{C_2}(-\Delta) \rightarrow \mathcal{I}_{C_1/C} \otimes N_{C} \rightarrow {T^1_C}_{|\Delta} \rightarrow 0
\end{equation}
\begin{equation}
\label{seq3}
0 \rightarrow {T_{\mathbb{P}^r}}_{|C_2}(-\Delta) \rightarrow {T_{\mathbb{P}^r}}_{|C} \rightarrow {T_{\mathbb{P}^r}}_{|C_1} \rightarrow 0.
\end{equation}
\end{lem}
\begin{proof}
Sequences (\ref{seq1}) and (\ref{seq2}) are straightforward generalizations of \cite{Sernesi}, Lemma 5.1, $(i)$ and $(ii)$. Sequence (\ref{seq3}) is obtained by tensorizing the exact sequence $0 \rightarrow \mathcal{O}_{C_2}(-\Delta) \rightarrow \mathcal{O}_C \rightarrow \mathcal{O}_{C_1} \rightarrow 0$ by the locally free sheaf ${T_{\mathbb{P}^r}}_{|C}$.
\end{proof}
Let $R \subset \mathbb{P}^r$ be a smooth rational curve of degree $d$. We say that $R$ is general if it is general as a point of the scheme parameterizing smooth rational curves of degree $d$ in $\mathbb{P}^r$.
\begin{lem}
\label{normalitangentiristretti}
Let $c$ be a positive integer and let $R,R_c \subset \mathbb{P}^r$ be, respectively, a rational normal curve and a smooth general rational curve of degree $r+c$. Then
$${T_{\mathbb{P}^r}}_{|R_c} \cong \mathcal{O}_{\mathbb{P}^1}(r+c+1)^{\oplus h} \oplus \mathcal{O}_{\mathbb{P}^1}(r+c+2)^{\oplus r-h} \hbox{ for some } h \leq r$$
$${T_{\mathbb{P}^r}}_{|R} \cong \mathcal{O}_{\mathbb{P}^1}(r+1)^{\oplus r}$$
$$N_{R} \cong \mathcal{O}_{\mathbb{P}^1}(r+2)^{\oplus r-1}$$
$$N_{R_c} \cong \mathcal{O}_{\mathbb{P}^1}(r+2+c)^{\oplus h'} \oplus \mathcal{O}_{\mathbb{P}^1}(r+3+c)^{\oplus r-1-h'} \hbox{ for some } h' \leq r-1.$$
\end{lem}
\begin{proof}
The first two isomorphisms follow from \cite{Ramella}, Theorem 2. The last two follow from \cite{Ran}, Theorem 6.1 (see also \cite{Ghione} for the case $r=3$).
\end{proof}
\end{section}
\begin{section}{The main result}
Let $\mathbb{G}$ be the (irreducible) family of nodal, irreducible, nondegenerate curves $\Gamma \subset \mathbb{P}^r$ with $r+1$ nodes and such that $\deg \Gamma=2r$, $p_a(\Gamma)=r+1$ and $\omega_{\Gamma}=\mathcal{O}_{\Gamma}(1)$.
\begin{lem}
\label{gamma}
Let $C_1,C_2 \subset \mathbb{P}^r$ be distinct rational normal curves intersecting at $r+2$ points. Then $C \doteq C_1 \cup C_2$ deforms to a curve $\Gamma \in \mathbb{G}$ such that $h^1(N'_{\Gamma})=0$ and $h^{0}({T_{\mathbb{P}^r}}_{|\Gamma})=(r+1)^2-1$.
\end{lem}
\begin{proof}
$C_1$ and $C_2$ intersect quasi-transversally at exactly $r+2$ points. Indeed there is at most one rational normal curve passing through $r+3$ points, counted with multiplicities (see \cite{EH}, Theorem 1), hence if $C_1$ and $C_2$ intersect each other in more than $r+2$ points, counted with multiplicities, then $C_1 =C_2$, contradiction.\\
Adjunction gives $\omega_C =\mathcal{O}_C(1)$, thus
in order to show that $C$ deforms to a curve $\Gamma \in \mathbb{G}$ it is sufficient to show that $h^1(N'_C)=0$. Let $\Delta \doteq C_1 \cap C_2$. By Lemma \ref{normalitangentiristretti} one has $h^{1}(N_{C_1})=0$ and $h^{1}(N_{C_2}(-\Delta))=0$. The cohomology sequences associated to (\ref{seq1}) and (\ref{seq2}) then give $h^1(N'_C)=0$.\\
Consider the Mayer-Vietoris sequence
$$0 \rightarrow {T_{\mathbb{P}^r}}_{|C} \rightarrow {T_{\mathbb{P}^r}}_{|C_1} \oplus {T_{\mathbb{P}^r}}_{|C_2} \xrightarrow{\alpha} {T_{\mathbb{P}^r}}_{|\Delta} \rightarrow 0.$$
Since the points of $\Delta$ are in linearly general position, the map $\alpha$ induces a surjective restriction map on global sections, thus $h^0({T_{\mathbb{P}^r}}_{|C})=h^{0}({T_{\mathbb{P}^r}}_{|C_i})=(r+1)^2-1$.\\
The statement then immediately follows from the upper semicontinuity of the cohomology.
\end{proof}
\begin{oss}
\label{oss1}
Note that, by construction, the curve $\Gamma$ can always be assumed to contain $r+3$ general points in $\mathbb{P}^r$.
\end{oss}
\begin{lem}
\label{r+3punti}
Let $\Gamma$ be as in Lemma \ref{gamma} and let $m$ be any positive integer. Then, for all positive integers $k \leq r+3$, there exist $m$ pairwise disjoint rational normal curves, each intersecting $\Gamma$ quasi-transversally at (exactly) $k$ general points of $\Gamma$.
\end{lem}
\begin{proof}
Let $C_1,C_2 \subset \mathbb{P}^r$ be as in Lemma \ref{gamma}. Let $S \doteq C_1 \cap C_2$, let $O \in S$, $P$ a general point in $C_1$ and $Q$ a general point in $C_2$. Set $A \doteq (S \smallsetminus \left\{O\right\}) \cup P$. Let $D_{A,Q}$ be the only rational normal curve containing $A \cup Q$. Then $C_1 \cap C_2 \cap D_{A,Q} = S \smallsetminus \left\{O\right\}$.\\
If $D_{A,Q}$ intersects $C_1$ in $r+3$ points, counted with multiplicities, then $D_{A,Q}=C_1$, which is a contradiction because $Q \notin C_1$, hence $D_{A,Q}$ intersects $C_1$ quasi-transversally at (exactly) $r+2$ points. The same argument with $P$ instead of $Q$ gives that $D_{A,Q}$ intersects $C_2$ quasi-transversally at (exactly) $r+2$ points. By deforming $C_1 \cup C_2$ to $\Gamma$, one obtains that there exists a rational normal curve $R$ intersecting $\Gamma$ quasi-transversally at (exactly) $r+3$ general points.\\
Let $T=\Gamma \cap R$ and let $T_{k}$ be any subscheme of $T$ of length $k$. Since $h^{1}(N_R(-T))=0$, one has that $h^{0}(N_R(-T_{k-1})) > h^{0}(N_R(-{T_k}))$ for all $2 \leq k \leq r+3$. By Lemma \ref{nc-t} it follows that, for all positive integers $k < r+3$, there exists a rational normal curve intersecting $\Gamma$ quasi-transversally at (exactly) $k$ general points.
We now have to prove that the curves can be assumed to be pairwise disjoint. For the case $k=r+3$, fix two rational normal curves $D$ and $D'$, both intersecting $\Gamma$ quasi-transversally at $r+3$ points, and such that $\Gamma \cap D \cap D'$ is a 0-dimensional scheme $W$ of $r+2$ points in linearly general position. Since $D$ and $D'$ are distinct, one has $D \cap D'=W$. The statement follows by deforming $D'$ to a $D''$ such that $\Gamma \cap D \cap D''=\emptyset$. Now assume $k<r+3$.
If $m=1$, then there is nothing to prove. Assume $m\ge 2$ and to have proven the existence of rational normal curves $Y_1,\dots ,Y_{m-1}$
such that $Y_i\cap \Gamma$ has length $k$, each $Y_i$ intersects quasi-transversally $\Gamma$ and $Y_i\cap Y_j =\emptyset$ for all $i\ne j$.
For each 0-dimensional subscheme $A\subset \Gamma$ of length $k$, let $E(A)$ be the set of all rational normal curves $Y\subset \mathbb {P}^r$ such
that $Y\cap \Gamma =A$ and $Y$ intersects quasi-transversally $\Gamma$. For a general $A$, the set $E(A)$
is non-empty and of dimension $h^{0}(N_Y(-A))=(r-1)(r+3-k)$. Fix a general $A\subset \Gamma$ of length $k$. In particular we assume
$A\cap (Y_i\cap \Gamma )=\emptyset$ for all $i=1,\dots ,m-1$. For
any $O\in Y_1\cup \cdots \cup Y_{m-1}$ set $E(A,O):= \{Y\in E(A):O\in Y\}$. If $\{O\}\cup A$ is not in linearly general position,
then $E(A,O)=\emptyset$. If $\{O\}\cup A$ is in linearly general position, then $E(A,O)$ has dimension
$(r-1)(r+2-k) =\dim (E(A)) -r+1$. Since $\dim (Y_1\cup \cdots \cup Y_{m-1}) = 1$, we get
the existence of $Y_m\in E(A)$ such that $Y_m\cap Y_i =\emptyset$ for all $i<m$.
\end{proof}
\begin{lem}
\label{R_cr+3+c}
Let $\Gamma$ be as in Lemma \ref{gamma} passing through $r+3$ general points $p_1,...,p_{r+3}$ in $\mathbb{P}^r$, let $c \leq r+3$ be a non-negative integer, let $r_1,...,r_c$ be general points in $\mathbb{P}^r$ and let $E_1,...,E_{c^{\prime}}$, $c^{\prime} \geq c$, be rational normal curves such that $E_i \cap E_j = \emptyset$ for all $i \neq j$ and, for $h=1,...,c$, one has $r_h \in E_h$. Let $E'=\Gamma \cup E_1 \cup \dots \cup E_{c^{\prime}} \subset \mathbb{P}^r$. Then there exists a smooth rational curve $R_c \subset \mathbb{P}^r$ of degree $r+c$ intersecting $E'$ quasi-transversally at (exactly) $p_1,...,p_{r+3},r_1,...,r_c$. Moreover, $R_c$ can be assumed to be a general smooth rational curve of degree $r+c$ through $p_1,...,p_{r+3},r_1,...,r_c$.
\end{lem}
\begin{proof}
We may find $\Gamma$ passing through $r+3$ general points of $\mathbb {P}^r$ by Remark \ref{oss1}.
Consider the $c$ lines $l_i=<p_{r+3-c+i},r_i>$, $i=1,...,c$, which are necessarily pairwise disjoint and intersecting $E'$ quasi-transversally due to the generality assumption on the points.\\
Let $q_i$ be a general point on $l_i$ for all $i$, let $T \doteq \left\{p_1,...,p_{r+3-c}\right\}$, and let $D \subset \mathbb{P}^r$ be the rational normal curve intersecting $\Gamma$ at $T$ and $l_i$ at $q_i$ for all $i$. By Lemma \ref{r+3punti} and the generality assumption on the $q_i$, one has that $D \cap E' =T$ and $D \cap l_i=q_i$ with quasi-transversal intersections.\\
Let $S \doteq \left\{p_1,...,p_{r+3},r_1,...,r_c\right\}$ and $W \doteq \left\{q_1,...,q_c\right\}$. We want to show that the reducible curve $D' \doteq D \cup l_1 \cup ... \cup l_c$ can be deformed to a smooth curve keeping the points of $S$ fixed.\\
By Lemma \ref{nc-t} it is sufficient to show that, for all $q_i \in W$, the map $H^{0}(N_{D'}(-S)) \rightarrow H^{0}(T^1_{q_i}(-S))$ is surjective and that $H^{1}(N_{D'}(-S))=(0)$.\\
From \cite{HH}, Corollary 3.2, there exist exact sequences
$$0 \rightarrow N_{l_i} \rightarrow {N_{D'}}_{|l_i} \rightarrow T^1_{q_i} \rightarrow 0$$
$$0 \rightarrow N_D \rightarrow {N_{D'}}_{|D} \rightarrow T^1_{W} \rightarrow 0.$$
The first one tensorized by $\mathcal{O}_{D'}(-S)$ gives $H^{1}({N_{D'}}(-S)_{|l_i})=(0)$ and the surjectivity of the map $H^{0}({N_{D'}}(-S)_{|l_i}) \rightarrow H^{0}(T^1_{q_i}(-S))$ for all $i$, while the second one tensorized by $\mathcal{O}_{D'}(-S-W)$ gives $H^1({N_{D'}}_{|D}(-S-W))=(0)$. We use, respectively, the fact that $H^{1}(N_{l_i}(-S))=(0)$ and $H^{1}(N_{D}(-S-W))=(0)$ by Lemma \ref{normalitangentiristretti}.\\
The exact sequence
$$0 \rightarrow {N_{D'}}_{|D}(-S-W) \rightarrow {N_{D'}}_{|D}(-S) \rightarrow {N_{D'}}_{|W}(-S) \rightarrow 0$$
then gives $H^1({N_{D'}}_{|D}(-S))=(0)$ and the surjectivity of the cohomology map\\ $H^{0}({N_{D'}}_{|D}(-S)) \rightarrow H^{0}({N_{D'}}_{|W}(-S))$.\\
From the surjecivity of the above considered maps, using the Mayer-Vietoris sequence
\begin{equation}
\label{nd'nd'dnd'li}
0 \rightarrow N_{D'} \rightarrow {N_{D'}}_{|D} \oplus \bigoplus_{i=1}^{c} {N_{D'}}_{|l_i} \rightarrow {N_{D'}}_{|W} \rightarrow 0
\end{equation}
tensorized by $\mathcal{O}_{D'}(-S)$, it follows that the maps $H^{0}(N_{D'}(-S)) \rightarrow H^{0}(T^1_{q_i}(-S))$ are surjective for all $i$. Moreover, from the fact that $H^{1}({N_{D'}}(-S)_{|l_i})=H^1({N_{D'}}_{|D}(-S))=(0)$ for all $i$, it follows that $H^1(N_{D'}(-S))=(0)$.\\
The last statement immediately follows from the fact that the scheme parameterizing smooth rational curves of degree $r+c$ through $S$ has dimension $h^{0}(N_{D'}(-S))$.
\end{proof}
Let $\Gamma$ be as in Lemma \ref{gamma} and let $c \leq r+3$ be a non-negative integer. By Lemma \ref{r+3punti}, there exist pairwise disjoint rational normal curves in $\mathbb{P}^r$ intersecting $\Gamma$ quasi-transversally at $r+2$ or $r+3$ points, say $X_i$, $i=1,...,a$ and $Y_j$, $j=1,...,b$, respectively. Assume $a \geq c$. By Lemma \ref{R_cr+3+c} there exists a general smooth rational curve $R_c$ of degree $r+c$ intersecting quasi-transversally $\Gamma$ at $r+3$ points, each of the curves $X_1,...,X_c$ at one point and not intersecting $X_i$, $i >c$ (if present) and $Y_j$ for all $j$. Define $Z=Z_{c,a,b} \doteq \Gamma \cup X_1 \cup ... \cup X_a \cup Y_1 \cup ... \cup Y_b \cup R_c$. Moreover, define $\Delta_{X_i} \doteq \Gamma \cap X_i$, $\Delta_{Y_j} \doteq \Gamma \cap Y_j, \Delta_{R_c} \doteq \left( \Gamma \cup X_1 \cup ... \cup X_c\right) \cap R_c$.
\begin{lem}
\label{nxi-deltai}
Notation as above, one has
$$h^{0}({T_{\mathbb{P}^r}}_{|R_c}(-\Delta_{R_c}))=h^{0}({T_{\mathbb{P}^r}}_{|X_i}(-\Delta_{X_i}))=h^{0}({T_{\mathbb{P}^r}}_{|Y_j}(-\Delta_{Y_j}))=0$$
and
$$h^1(N_{R_c}(-\Delta_{R_c}))=h^{1}(N_{X_i}(-\Delta_{X_i}))=h^{1}(N_{Y_j}(-\Delta_{Y_j}))=0$$ for all $0 \leq c$, $i=1,...,a$, $j=1,...,b$.
\end{lem}
\begin{proof}
It is an immediate consequence of Lemma \ref{normalitangentiristretti}.
\end{proof}
\begin{prop}
\label{zgoodproperties}
Let $Z$ be as above. Then one has $h^{0}({T_{\mathbb{P}^r}}_{|Z})=(r+1)^2-1$ and $H^{1}(N'_Z)=(0)$.
\end{prop}
\begin{proof}
Let $D_{0} \doteq \Gamma$, $D_i \doteq X_i$, $i=1,...,a$, $D_j \doteq Y_j$, $j=a+1,...,a+b$, $D_{a+b+1} \doteq R_c$.
Let $Z_h \doteq D_{0} \cup ... \cup D_h$, $h=0,...,a+b+1$. Note that $Z_{a+b+1}=Z$.\\
Sequence (\ref{seq3}) of Lemma \ref{sequences} gives
\begin{equation}
\label{zh}
0 \rightarrow {T_{\mathbb{P}^r}}_{|D_{h+1}}(-\Delta_{D_{h+1}}) \rightarrow {T_{\mathbb{P}^r}}_{|Z_{h+1}} \rightarrow {T_{\mathbb{P}^r}}_{|Z_{h}} \rightarrow 0
\end{equation}
for all $h=0,...,a+b$.\\
By Lemma \ref{gamma} one has $h^{0}({T_{\mathbb{P}^r}}_{|Z_0})=(r+1)^2-1$. Since from Lemma \ref{nxi-deltai} one has $H^{0}({T_{\mathbb{P}^r}}_{|D_{h+1}}(-\Delta_{D_{h+1}}))=(0)$ for all $h=0,...,a+b$, by inductively adding components to $Z_h$ in sequence (\ref{zh}) one obtains $h^{0}({T_{\mathbb{P}^r}}_{|Z})=(r+1)^2-1$.\\
Sequence (\ref{seq2}) of Lemma \ref{sequences} gives
$$0 \rightarrow N_{D_{h+1}}(-\Delta_{D_{h+1}}) \rightarrow \mathcal{I}_{Z_{h}/Z_{h+1}} \otimes N_{Z_{h+1}} \rightarrow {T^1_{Z_{h+1}}}_{|\Delta_{D_{h+1}}} \rightarrow 0$$
for all $h=0,...,a+b$. By Lemma \ref{nxi-deltai}, it follows that $H^{1}(\mathcal{I}_{Z_{h}/Z_{h+1}} \otimes N_{Z_{h+1}})=(0)$ for all $h=0,...,a+b$. Sequence (\ref{seq1}) of Lemma \ref{sequences} gives
$$0 \rightarrow \mathcal{I}_{Z_{h}/Z_{h+1}} \otimes N_{Z_{h+1}} \rightarrow N'_{Z_{h+1}} \rightarrow N'_{Z_h} \rightarrow 0.$$
Since $H^{1}(N'_{Z_0})=(0)$ by Lemma \ref{gamma}, it follows that $H^{1}(N'_Z)=(0)$.
\end{proof}
We now have all the elements to prove our main result.

\begin{teo}\label{main}
Let $r \geq 3$ and $d \geq (2r+2)r$ be integers. For all integers $g, \delta, i$ such that
$d+\lfloor \frac{d}{r} \rfloor-r \leq g \leq d+ \lfloor \frac{d}{r} \rfloor$,
$1 \le i \le \lfloor \frac{d}{r} \rfloor - 2$ and $(r+3)(i-1) \leq \delta \leq g+i-1$,
there exists a connected reduced nondegenerate nodal curve $C^i_{\delta} \subset \mathbb{P}^r$ of arithmetic genus $g$ with
$\delta$ nodes and $i$ irreducible components (all but one smooth and rational), such that $[C^i_{\delta}]$ is a smooth
point of an irreducible component $V \subset \mathscr{V}^{r}_{d,g, \delta}$ having the expected number of moduli.
\end{teo}

\begin{proof}
We construct the curves $C^i_{\delta}$ as deformations of the reducible curves $Z=Z_{c,a,b}$ constructed above. One has $\deg(Z)=(a+b+3)r+c$ and $p_a(Z)=(a+b+2)(r+1)+b+c+1$. Let $d=\deg(Z)$ and let $0 \leq c<r$ be the integer such that $d \equiv c \mod r$. Since $d \geq (2r+2)r$ and $c < r$, it follows that $a+b \geq 2r-1$. Let $0 \leq b \leq r$ be the integer such that $g \equiv b+c+1 \mod r+1$. One has that $p_a(Z)=\left(\frac{d-c}{r}-1\right)(r+1)+b+c+1=d+\frac{d}{r}-\frac{c}{r}-r+b$. Since $a+b \geq 2r-1$ and $\frac{d}{r}-\frac{c}{r}=\lfloor \frac{d}{r} \rfloor$, it follows that one can find a pair $(a,b)$ with $a \geq c$ such that $p_a(Z)=g$ for all $d+\lfloor \frac{d}{r} \rfloor -r \leq g \leq d+ \lfloor \frac{d}{r} \rfloor $.\\
We have thus showed the existence of a reduced nodal curve $Z$ whose degree equals $d$ and whose arithmetic genus equals $g$. Since
$H^{1}(N'_Z)=(0)$ by Proposition \ref{zgoodproperties}, fixed any integer $i$ with $1 \leq i \leq a+b+1=\lfloor \frac{d}{r} \rfloor -2$, we can deform $Z$ to a nodal curve with exactly $i$ irreducible components by preserving all nodes on $i-1$ (smooth and rational) irreducible components of $X_1 \cup ... \cup X_a \cup Y_1 \cup ... \cup Y_b$ (hence at most $(r+3)(i-1)$ nodes) and by smoothing at least one node for all the remaining irreducible components of $X_1 \cup ... \cup X_a \cup Y_1 \cup ... \cup Y_b \cup R_c$. The fact that $V$ has the expected number of moduli follows from Proposition \ref{zgoodproperties} applied to Propositions \ref{condizioniperexpected} and \ref{condizioniequivalentialinearmentenormaleerangomassimo}, and the upper semicontinuity of the cohomology.
\end{proof}
The following result proves that, if $\Gamma \subset \mathbb{P}^r$ is a connected reduced nondegenerate nodal curve such that $h^{0}({T_{\mathbb{P}^r}}_{|\Gamma})=(r+1)^2-1$ and $h^{1}(N'_{\Gamma})=0$, then there are connected reduced nondegenerate nodal curves $\Gamma'$ in $\mathbb{P}^r$ with degree and arithmetic genus pretty close to $\deg(\Gamma)$ and $p_a(\Gamma)$, respectively, and whose restricted tangent bundle and equisingular normal sheaf have the same properties, so that $\Gamma'$ and each of its partial smoothings, if stable, are parameterized by a smooth point of an irreducible component of Hilbert schemes $\mathscr{V}^{r}_{\deg{\Gamma'},p_a(\Gamma'),\delta}$ having the expected number of moduli.
\begin{lem}
\label{lemmar+2punti}
Let $S \subset \mathbb{P}^r$ be a 0-dimensional scheme of $r+2$ points in linearly general position. For each $O \in S$, let $V_O$ be a family of lines passing through $O$, of dimension less or equal than $r-2$. Fix a locally closed subset $W \subset \mathbb{P}^r \smallsetminus S$ such that $\dim W \leq r-2$. Then there exists a rational normal curve $R \subset \mathbb{P}^r$ such that $S \subset R$, $R \cap W = \emptyset$ and for each $O \in S$ the tangent line $T_{O}R$ of $R$ at $O$ does not belong to $V_{O}$.
\end{lem}
\begin{proof}
Let $\mathcal{H}(S)$ be the set of all rational normal curves in $\mathbb{P}^r$ containing $S$. $\mathcal{H}(S)$ is a nonempty algebraic variety of dimension $r-1$. For each $O \in S$ and each $l \in V_{O}$, let $V_{l,O}$ be a tangent vector of $l$ at $O$. Since $S \cup V_{l,O}$ is a $0$-dimensional scheme of degree $r+3$, it is contained in at most one rational normal curve, say $D_{l,O}$. Let $\mathcal{B}$ be the set of all rational normal curves $D_{l,O}$ such that $O \in S$ and $l \in V_{O}$. We have $\dim \mathcal{B} \leq r-2$.\\
For each $p \in \mathbb{P}^r \smallsetminus S$, there is at most one rational normal curve $D_{p}$ containing $p$ and $S$. Let $\mathcal{F}$ be the set of all rational normal curves $D_{p}$, $p \in W$. Since $\dim \mathcal{F} \leq \dim W \leq r-2$, we have $\mathcal{H}(S) \supsetneq \mathcal{B} \cup \mathcal{F}$.
\end{proof}
\begin{teo}
Let $\Gamma \subset \mathbb{P}^r$ be a connected reduced nondegenerate nodal curve with $\eta$ nodes such that $h^{0}({T_{\mathbb{P}^r}}_{|\Gamma})=(r+1)^2-1$ and $h^{1}(N'_{\Gamma})=0$. Then for all integers $0 \leq \delta \leq r+2+\eta$, there is a connected reduced nondegenerate nodal curve $\Gamma'_{\delta} \subset \mathbb{P}^r$ with $\delta$ nodes such that $\deg(\Gamma'_{\delta})=\deg{\Gamma}+r$, $p_a(\Gamma'_{\delta})=p_a(\Gamma)+r+1$, $h^{0}({T_{\mathbb{P}^r}}_{|\Gamma'_{\delta}})=(r+1)^2-1$ and $h^{1}(N'_{\Gamma'_{\delta}})=0$. In particular, if $\Gamma'_{\delta}$ is stable, then it is parameterized by a smooth point of an irreducible component $V \subset \mathscr{V}^{r}_{\deg{\Gamma'_{\delta}},p_a(\Gamma'_{\delta}),\delta}$ having the expected number of moduli.
\end{teo}
\begin{proof}
Let $S$ be a 0-dimensional subscheme of $\Gamma$ of $r+2$ smooth points in linearly general position, let, for each $O \in S$, $V_{O} \doteq T_{O}\Gamma$ i.e. the tangent line to $\Gamma$ at $O$, and let $W \doteq \Gamma \smallsetminus S$. By applying Lemma \ref{lemmar+2punti} with $S$, $V_{O}$, $W$ as above, one obtains that there exists a rational normal curve $R \subset \mathbb{P}^r$ intersecting $\Gamma$ quasi-transversally at (exactly) $r+2$ smooth points. Let $\Gamma' \doteq \Gamma \cup R$. The same computations done in the proof of Proposition \ref{zgoodproperties} give $h^{0}({T_{\mathbb{P}^r}}_{|\Gamma'})=(r+1)^2-1$ and $h^{1}(N'_{\Gamma'})=0$, hence by semicontinuity the same properties hold for $\Gamma'_{\delta}$ for any $0 \leq \delta \leq r+2+\eta$. Propositions \ref{condizioniperexpected} and \ref{condizioniequivalentialinearmentenormaleerangomassimo} prove the last claim.
\end{proof}
\end{section}

\vspace{0.3cm}

\noindent
Edoardo Ballico \newline
Dipartimento di Matematica \newline
Universit\`a di Trento \newline
Via Sommarive 14 \newline
38123 Trento, Italy. \newline
E-mail address: ballico@science.unitn.it

\vspace{0.3cm}

\noindent
Luca Benzo \newline
Dipartimento di Matematica \newline
Universit\`a di Trento \newline
Via Sommarive 14 \newline
38123 Trento, Italy. \newline
E-mail address: luca.benzo@unitn.it

\vspace{0.3cm}

\noindent
Claudio Fontanari \newline
Dipartimento di Matematica \newline
Universit\`a di Trento \newline
Via Sommarive 14 \newline
38123 Trento, Italy. \newline
E-mail address: fontanar@science.unitn.it

\end{document}